\documentclass[graybox]{svmult}

\usepackage{type1cm}        
\usepackage{makeidx}         
\usepackage{graphicx}        
\usepackage{multicol}        
\usepackage[bottom]{footmisc}

\usepackage{newtxtext}       %
\usepackage[varvw]{newtxmath}       

\makeindex             

\newcommand{\dist}{\mathsf{d}}

\newcommand{\meas}{\mathfrak{m}}
\newcommand{\bist}{\mathsf{b}}

\newcommand{\Vol}{\mathop{\mathrm{Vol}}}

\newcommand{\eps}{\varepsilon}
\newcommand{\di}{\mathop{}\!\mathrm{d}}
\newcommand{\haus}{\mathscr{H}}
\newcommand{\Ch}{{\sf Ch}}

\DeclareMathOperator{\RCD}{RCD}

\begin{document}

\title*{Almost rigidity results of Green functions with non-negative Ricci curvature}
\author{Shouhei Honda}
\institute{Shouhei Honda \at  Mathematical Institute, Tohoku University, \email{shouhei.honda.e4@tohoku.ac.jp}
}
%
%
\maketitle


\abstract{This short note provides a survey on rigidity and almost rigidity results of Green functions in a non-smooth setting, obtained in \cite{HP}. We also make some observation on the Cheeger-Yau inequality on RCD spaces of non-negative Ricci curvature with applications.}

\section{Introduction}\label{sec1}
\subsection{Green function and smoothed distance function}
The \textit{Green function} $G^{\mathbb{R}^n}$ on the Euclidean space $\mathbb{R}^n$ of dimension  $n \ge 3$ is defined by 
\begin{equation}\label{eq:01}
G^{\mathbb{R}^n}(x,y)=G^{\mathbb{R}^n}(|x-y|)=\frac{|x-y|^{2-n}}{n(n-2)\omega_n}
\end{equation}
as a smooth function on $\mathbb{R}^n \times \mathbb{R}^n \setminus \mathrm{diag}(\mathbb{R}^n)$,
where $\mathrm{diag}(\mathbb{R}^n)$ denotes the diagonal part, namely $\mathrm{diag}(\mathbb{R}^n):=\{(x,x) \in \mathbb{R}^n \times \mathbb{R}^n | x \in \mathbb{R}^n\}$, and $\omega_n:=\pi^{\frac{n}{2}}\Gamma (\frac{n}{2}+1)^{-1}$ is the volume of a unit ball in $\mathbb{R}^n$.
It follows from a direct calculation that the Green function solves the equation:
\begin{equation}
\Delta^{\mathbb{R}^n} u=-\delta_x
\end{equation}
as measures, where $\Delta^{\mathbb{R}^n}$ denotes the Laplacian on $\mathbb{R}^n$ and $\delta_x$ is the Dirac measure at $x$. 

Let us denote by $(M^n, g)$ a complete Riemannian manifold of dimension $n$ with non-negative Ricci curvature $\mathrm{Ric}^g \ge 0$.
A result in \cite{Varopoulos} states that 
the existence of a global positive Green function 
 is equivalent to verifying the following:
\begin{equation}\label{nonpara}
\int_1^\infty\frac{r}{\Vol^g B_r^g(x)}\ \di r<\infty, \quad \forall x\in M^n,
\end{equation}
where $B_r^g(x)$ denotes the open ball centered at $x$ of radius $r$ with respect to the Riemannian distance $\dist^g$ by $g$, and $\Vol^g$ denotes the Riemannian volume measure by $g$.
In the sequel, we assume that $(M^n, g)$ is \textit{non-parabolic}, namely (\ref{nonpara}) holds (thus $n\ge 3$).

Then it is well-known that the following asymptotic behavior for the minimal Green function $G_x^g:=G^g(x, \cdot)$ at the pole $x \in M^n$ holds as $\dist_x^g\to 0^+$, where $\dist_x^g:=\dist^g(x, \cdot)$:
\begin{equation}\label{eq:03}
G_x^g=\frac{(\dist_x^g)^{2-n}}{n(n-2)\omega_n}+o((\dist_x^g)^{2-n}),
\end{equation}
namely, the Green function behaves like in the  Euclidean case at small scale (because the tangent space at $x$ is trivially isometric to $\mathbb{R}^n$).

On the other hand, let us recall that the distance function $\dist^g_x$ from a point $x \in M^n$ is smooth outside $x$ and the cut locus. In order to regularize the distance function $\dist^g_x$ to be smooth over the cut locus, we will be able to use the Green function as follows.

\begin{definition}[Smoothed distance function]\label{defsmooth}
Define a smooth function $\bist^g$ on $M^n \times M^n \setminus \mathrm{diag}(M^n)$ by
\begin{equation}\label{smooth dist def}
\bist^g:=\left(G^g\right)^{\frac{1}{2-n}}
\end{equation}
and call $\bist^g$ the \textit{smoothed distance function}. Denote $\bist^g_x:=\bist^g(x, \cdot)$.
\end{definition}

Note that thanks to (\ref{eq:01}), if $(M^n, g)$ is isometric to the Euclidean space $(\mathbb{R}^n, \dist_{\mathrm{Euc}})$ of dimension $n$, then $\bist_x^g$ coincides with the distance function from $x$, up to multiplying a positive dimensional constant:
\begin{equation}
\bist_x^g(y)=\bist^g(x, y)=c_n \omega_n^\frac{1}{n-2} |x-y|,
\end{equation}
where
\begin{equation} 
c_n:=\left(n(n-2)\right)^{\frac{1}{n-2}}.
\end{equation}

It should be emphasized that the smoothed distance function (or its variant) played important roles in recent dvelopments on geometric analysis, for instance, see \cite{CM, AFM, ChLi}.

\subsection{Results}\label{smooth result}
In order to introduce a main result of \cite{HP}, let us recall a rigidity result on the smoothed distance function $\bist_x^g$ obtained in \cite{C12}.
	\begin{theorem}[Sharp gradient estimate and rigidity, \cite{C12}]\label{Colding}
	Let $(M^n,g)$ be a complete non-parabolic (namely (\ref{nonpara}) holds) Riemannian manifold of dimension $n \ge 3$ with non-negative Ricci curvature $\mathrm{Ric}^g \ge 0$. Then for any fixed $x \in M^n$:
	\begin{enumerate}
		\item{(Sharp gradient estimate)} we have for any $y \in M^n \setminus \{x\}$
		\begin{equation}\label{sharp constant2}
		|\nabla \bist_x^g|(y)\le c_n \omega_n^\frac{1}{n-2}:
		\end{equation}
		\item{(Rigidity)} $(M^n,g)$ is isometric to the $n$-dimensional Euclidean space $(\mathbb{R}^n, \dist_{\mathrm{Euc}})$ with $\bist_x^g=c_n \omega_n^\frac{1}{n-2} \dist_x^g$ if the equality of (\ref{sharp constant2}) holds for some $y \in M^n \setminus \{x\}$.
	\end{enumerate}
\end{theorem}
	See also \cite{HXY} for a related work.
	From the theorem above, it is natural to ask the following question (Q).
	\begin{itemize}
	\item[(Q)]\label{Q} \,\,If $|\nabla \bist_x^g|(y)$ is close to $c_n \omega_n^\frac{1}{n-2}$ at some point $y \in M^n \setminus \{x\}$, then can we conclude that the manifold is pointed Gromov-Hausdorff (pGH) close to $\mathbb{R}^n$? 
\end{itemize}
In connection with this question (Q), it is worth mentioning that we have
\begin{equation}\label{as shar}
|\nabla \bist_x^g|(y) \to  c_n\omega_n^\frac{1}{n-2}
\end{equation}
whenever $y \to x$ because of (\ref{eq:03}). Therefore in order to give a positive answer to the question (Q), we need to find an additional assumption on $y$.

We are now in a position to introduce a main result along this direction, obtained in \cite{HP}. 
We denote by $\dist_{\mathrm{pmGH}}$ a fixed distance metrizing the pointed measured Gromov-Hausdorff (pmGH) convergence.

\begin{theorem}[Almost rigidity, Part I: \cite{HP}]\label{thm:smooth almost}
For any integer $n \ge 3$, all $0<\eps<1$, $0<r<R$, $1 \le p <\infty$ and $\phi \in L^1([0,\infty), \mathcal{L}^1)$ there exists $\delta:=\delta(n, \eps, r, R, p, \phi)>0$ such that if 
a complete Riemannian manifold $(M^n, g)$ of dimension $n$ with non-negative Ricci curvature $\mathrm{Ric}^g \ge 0$ satisfies
\begin{equation}\label{s8syshsbb}
\frac{s}{\Vol^g B_s^g(x)} \le \phi(s),\quad \text{for $\mathcal{L}^1$-a.e. $s \in [1, \infty)$}
\end{equation}
for some $x \in M^n$ and that
\begin{equation}\label{almost max point}
c_n \omega_n^\frac{1}{n-2}-|\nabla \bist_x^g|(y) \le \delta 
\end{equation}
holds for some $y \in B_R^g(x) \setminus B_r^g(x)$, then we have 
\begin{align}\label{r lower}
\dist_{\mathrm{pmGH}}\left( (M^n, \dist^g, \mathrm{Vol}^g, x), (\mathbb{R}^n, \dist_{\mathbb{R}^n}, \mathcal{L}^n, 0_n) \right)\le \eps
\end{align}
and 
\begin{equation}\label{impsosirbb}
\left\| \bist_x^g-c_n \omega_n^\frac{1}{n-2}\dist_x^g\right\|_{L^{\infty}(B_R^g(x))}+ \left\| \bist_x^g-c_n \omega_n^\frac{1}{n-2}\dist_x^g\right\|_{H^{1, p}(B_R^g(x))}\le \eps
\end{equation}
in particular
\begin{equation}\label{elpest222}
\|c_n \omega_n^\frac{1}{n-2}-|\nabla \bist_x^g| \|_{L^p(B_R^g(x))} \le \eps.
\end{equation}
\end{theorem}
Note that though Theorem \ref{thm:smooth almost} is stated in the smooth category, this is also valid in a non-smooth category, so-called (non-collapsed) RCD spaces. See \cite{HP} for the details.

In this paper we also provide a related almost rigidity (Theorem \ref{almost rigi}) and a Cheeger-Yau type comparison result for the heat kernels (Proposition \ref{CheegerYau}) in the non-smooth category. They might be useful for readers, though they are also direct consequences of previously known results.

\section{Proof}\label{sec2}
Firstly let us  recall the proof of Theorem \ref{Colding}.
\subsection{Proof of Theorem \ref{Colding}}\label{subColding}
Fix a complete non-parabolic Riemannian manifold $(M^n, g)$ of dimension $n\ge 3$ with non-negative Ricci curvature $\mathrm{Ric}^g \ge 0$. 

Let us prove the sharp gradient estimate (\ref{sharp constant2}) for the smoothed distance function $\bist^g_x:M^n \to [0, \infty)$ from a point $x \in M^n$, defined in (\ref{smooth dist def}). 
Applying the Bochner formula for $(\bist_x^g)^2$, we know that $G_x^g|\nabla \bist^g_x|^2$ is subharmonic because of the following calculation:
\begin{align}\label{ss8sshhshhs}
\Delta^g\left(G_x^g|\nabla \bist^g_x|^2\right)&=\frac{1}{2}\left( \left|\mathrm{Hess}^g_{(\bist_x^g)^2}-\frac{\Delta^g(\bist_x^g)^2}{n}g\right|^2+\mathrm{Ric}^g\left(\nabla (\bist_x^g)^2, \nabla (\bist_x^g)^2\right)\right)\left(\bist_x^g\right)^{-n}\nonumber \\
& \ge 0,
\end{align}
due to \cite[Lemma 2.7]{C12}.

 Therefore considering a large anulus, the weak maximum principle yields that it is enough to show the sharp gradient estimate (\ref{sharp constant2}) near $x$ because we can also get a decay estimate on $G_x^g$.
Then recalling (\ref{as shar}), we can reach the sharp gradient estimate (\ref{sharp constant2}) on $M^n \setminus \{x\}$.

Secondly let us prove the rigidity result. Assume $|\nabla \bist_x^g|(y)=c_n \omega_n^{\frac{1}{n-2}}$ for some $y \in M^n \setminus \{x\}$. Recalling the proof above, the strong maximum principle allows us to conclude that $|\nabla \bist_x^g|$ must be a constant, namely $|\nabla \bist_x^g|\equiv c_n \omega_n^{\frac{1}{n-2}}$ on $M^n \setminus \{x\}$. Thus by (\ref{ss8sshhshhs}), we know that $\mathrm{Hess}_{(\bist_x^g)^2}$ is equal to $g$ up to multiplying a positive dimensional constant because $\Delta^g (\bist_x^g)^2=2n|\nabla \bist_x^g|^2$ is also a dimensional constant. This implies that $(M^n, g)$ is isometric to a metric cone with the warping function $(\bist_x^g)^2$. Since any smooth metric cone must be isometric to a Euclidean space because of the scale invariance of the metric cone, the proof of Theorem \ref{Colding} is completed. 
\subsection{$\RCD$ spaces}
In order to prove Theorem \ref{thm:smooth almost}, we need to provide a generalization of Theorem \ref{Colding} to non-smooth spaces with Ricci curvature bounded below, so-called $\RCD(0, N)$ spaces. 
Therefore let us give a brief introduction on $\RCD(K, N)$ spaces. We refer \cite{A, GIGLI, STURM} as recent nice surveys.
\subsubsection{Definition}
Fix a complete separable metric space $(X, \dist)$ and a Borel measure $\meas$ on $X$ with full support, namely $\meas(U)>0$ holds for any non-empty open subset $U$ of $X$. 
Moreover we assume that $\meas$ is finite for any bounded Borel subset. Such a triple $(X, \dist, \meas)$ is called a \textit{metric measure space}.

The \textit{Cheeger energy} $\Ch:L^2(X, \meas) \to [0, \infty]$ is defined by
\begin{equation}\label{cheegert}
\Ch(f)=\inf_{\{f_i\}_i}\left\{ \liminf_{i\to \infty}\frac{1}{2}\int_X\left(\mathrm{Lip}f_i(x)\right)^2\di \meas(x) \right\},
\end{equation}
where the infimum $\{f_i\}_i$ runs over all bounded Lipschitz $L^2$-functions with  $\|f_i - f\|_{L^2} \to 0$ as $i \to \infty$, and $\mathrm{Lip}f(x)$ denotes the local slope of $f$ at $x$ defined by:
\begin{equation}
\mathrm{Lip}f(x):=\limsup_{y \to x}\frac{|f(x)-f(y)|}{\dist(x, y)}.
\end{equation}
The \textit{Sobolev space} $H^{1,2}=H^{1,2}(X, \dist, \meas)$ is defined by the finiteness domain of $\Ch$ and then it is a Banach space equipped with the norm $\|f\|_{H^{1,2}}:=(\|f\|_{L^2}^2+2\Ch(f))^{\frac{1}{2}}$. 

For any $f \in H^{1,2}$, considering the collection $R(f)$ of all functions in $L^2(X, \meas)$ larger than or equal to a weak $L^2$-limit of $\mathrm{Lip}f_i(x)$ appeared in (\ref{cheegert}), we know that $R(f)$ is a closed convex non-empty subset in $L^2(X, \meas)$. Denoting by $|\nabla f|$ the unique element of $R(f)$ with the smallest $L^2$-norm, called the \textit{minimal relaxed slope} of $f$, we have

\begin{equation}
\Ch(f)=\frac{1}{2}\int_X|\nabla f|^2\di \meas.
\end{equation}

The metric measure space $(X, \dist, \meas)$ is called \textit{infinitesimally Hilbertian} (IH) if the Sobolev space $H^{1,2}$ is a Hilbert space. Assume that $(X, \dist, \meas)$ is IH below.

For all $f, h \in H^{1,2}$,
\begin{equation}
\langle \nabla f, \nabla h\rangle := \lim_{t \to 0}\frac{|\nabla (f+th)|^2-|\nabla f|^2}{2t}
\end{equation}
determines an $L^1$-function on $X$. We denote by $D(\Delta)$ the domain of the Laplacian, namely, $f \in D(\Delta)$ holds if and only if $f \in H^{1,2}$ and there exists a unique $\phi \in L^2(X, \meas)$, denoted by $\Delta f$, such that
\begin{equation}
\int_X\langle \nabla f, \nabla h\rangle \di \meas=-\int_X\phi h\di \meas,\quad \text{for any $h \in H^{1,2}$.}
\end{equation}

We are now in a position to introduce the precise definition of $\RCD(K, N)$ spaces.
\begin{definition}[$\RCD(K, N)$ space]
Under the same notation as above (recall that we assume IH for $(X, \dist, \meas)$), the metric measure space $(X, \dist, \meas)$ is called an $\RCD(K, N)$ \textit{space} for some $K \in \mathbb{R}$ and some $N \in [1, \infty]$ if the following three conditions are satisfied:
\begin{enumerate}
\item there exist a positive constant $C>1$ and a point $x \in X$ such that 
\begin{equation}
\meas(B_r(x))\le C\exp (Cr^2),\quad \text{for any $r>1$:}
\end{equation}
\item if a Sobolev function $f \in H^{1,2}$ satisfies $|\nabla f|(x) \le 1$ for $\meas$.a.e. $x \in X$, then $f$ has a $1$-Lipschitz representative:
\item we have the Bochner inequality:
\begin{equation}\label{s8sabasy}
\frac{1}{2}\int_X|\nabla f|^2 \cdot \Delta \phi \di \meas \ge \int_X\phi \left( \frac{(\Delta f)^2}{N} +\langle \nabla \Delta f, \nabla f\rangle +K|\nabla f|^2 \right)\di \meas
\end{equation}
for all $\phi \in L^{\infty}(X, \meas) \cap D(\Delta)$ with $\Delta \phi \in L^{\infty}(X, \meas)$, and $f \in D(\Delta)$ with $\Delta f \in H^{1,2}$.
\end{enumerate}
\end{definition}
Roughly speaking, $(X, \dist, \meas)$ is an $\RCD(K, N)$ space if Ricci curvature is bounded below by $K$ and dimension is bounded above by $N$ in a synthetic sense. Note that $N$ is not necessarily to be an integer. 
See \cite{Sturm06, Sturm06b, LottVillani, AmbrosioGigliSavare14, AmbrosioGigliMondinoRajala, ErbarKuwadaSturm, Gigli1, AmbrosioMondinoSavare, CavallettiMilman,  LI} for other equivalent definitions.

Let us end this subsection by giving a special class of $\RCD(K, N)$ spaces introduced in \cite{DePhillippisGigli}.
\begin{definition}[Non-collapsed $\RCD(K, N)$ space]
An $\RCD(K, N)$ space $(X, \dist, \meas)$ for some $K \in \mathbb{R}$ and some finite $N \in [1, \infty)$ is said to be \textit{non-collapsed} if $\meas$ coincides with the $N$-dimensional Hausdorff measure $\haus^N$.
\end{definition}

\subsubsection{Non-parabolic $\RCD(0, N)$ space}
The main purpose of this subsection is to provide fundamental results on $\RCD(0, N)$ spaces, without the proofs. 

Let $(X, \dist, \meas)$ be an $\RCD(0, N)$ space for some finite $N \ge 1$. Such typical examples include \textit{$N$-metric measure cones} over $\RCD(N-2, N-1)$ spaces, due to \cite{Ketterer2}. 
For example, a metric measure space for all $C>0$ and $N \in [1, \infty)$:
\begin{equation}\label{halfline}
\left([0, \infty), \dist_{\mathrm{Euc}}, Cx^{N-1}\di x\right)
\end{equation}
is the $N$-metric measure cone over a single point, thus it is an $\RCD(0, N)$ space.

Recall that the Bochner inequality (\ref{s8sabasy}) for $(X, \dist, \meas)$ means:
\begin{equation}\label{isnshshshs}
\frac{1}{2}\Delta|\nabla f|^2\ge \frac{(\Delta f)^2}{N}+\langle \nabla \Delta f, \nabla f\rangle 
\end{equation}
in a weak sense. In connection with this,
it should  be emphasized that the Bochner inequality including the (well-defined) \textit{Hessian} term:
\begin{equation}\label{bochner}
\frac{1}{2}\Delta |\nabla f|^2\ge |\mathrm{Hess}_f|^2+\langle \nabla \Delta f, \nabla f\rangle
\end{equation}
also holds in a weak sense. See \cite{Gigli}. This will play key roles to get structure results on $(X, \dist, \meas)$.

Fundamental properties on $(X, \dist, \meas)$ include the \textit{Bishop-Gromov inequality} \cite{Sturm06b, LottVillani}: 
\begin{equation}\label{bgbgbg}
\frac{\meas(B_r(x))}{r^N}\ge \frac{\meas(B_s(x))}{s^N},\quad \text{for all $x \in X$ and $0<r<s$.}
\end{equation}
In particular the \textit{$N$-volume density} $\nu_x=\nu_x^N$ at a point $x \in X$:
\begin{equation}\label{vol density}
\nu_x:=\lim_{r \to 0^+}\frac{\meas (B_r(x))}{r^N} \in (0, \infty]
\end{equation}
is well-defined. It is known that if $\nu_x$ is finite, then any tangent cone at $x$ is isomorphic to the $N$-metric measure cone over an $\RCD(N-2, N-1)$ space because of \cite{DG} (see also \cite{GV}). 
Note that if $(X, \dist, \meas)$ is non-collapsed, then $N$ is an integer, the \textit{Bishop inequality} holds in the sense: 
\begin{equation}\label{asasj8a8s}
\nu_x \le \omega_N
\end{equation} 
and the rigidity is also satisfied; the equality in (\ref{asasj8a8s}) holds if and only if $x$ is an $N$-regular point. See \cite{DePhillippisGigli}.

Let us introduce the \textit{non-parabolicity} of $(X, \dist, \meas)$ as in the smooth case \cite{Varopoulos}.  See \cite{BrueSemola, CMon}.
\begin{definition}[Non-parabolicity]
We say that an $\RCD(0, N)$ space $(X, \dist, \meas)$ for some finite $N \ge 1$ is \textit{non-parabolic} if 
\begin{equation}\label{s9snsnbs}
\int_1^{\infty}\frac{r}{\meas(B_r(x))}\di r<\infty
\end{equation}
holds for some (or equivalently any) $x \in X$.
\end{definition}
Note that (\ref{s9snsnbs}) imples $N>2$ because of (\ref{bgbgbg}).

We can also define the \textit{Green}/\textit{smoothed distance functions} as follows, where they are well-defined because of a \textit{Gaussian estimate} for the \textit{heat kernel} $p=p^X$ of $(X, \dist, \meas)$ established in \cite{JiangLiZhang}. 
\begin{definition}[Green/smoothed distance functions]
Assume that $(X, \dist, \meas)$ is a non-parabolic $\RCD(0, N)$ space for some finite $N >2$. Then the \textit{Green function} $G=G^X:X \times X \setminus \mathrm{diag}(X) \to (0, \infty)$ of $(X, \dist, \meas)$ is defined by
\begin{equation}
G(x,y):=\int_0^{\infty}p(x,y,t)\di t.
\end{equation}
We also define the \textit{smoothed} ($N$-)\textit{distance function} $\bist_x$ from a point $x \in X$ by;
\begin{equation}\label{smdist}
\bist_x:=G(x, \cdot)^{\frac{1}{2-N}}.
\end{equation}
\end{definition}
Then we are now in a position to introduce a non-smooth analogue of Theorem \ref{Colding} which plays a key role in the proof of Theorem \ref{thm:smooth almost}. 
\begin{theorem}[Sharp gradient estimate and rigidity]\label{rcdsharp}
Assume that $(X, \dist, \meas)$ is  a non-parabolic $\RCD(0, N)$ space for some finite $N >2$. Then for any point $x \in X$ with $\nu_x<\infty$:
\begin{enumerate}
		\item{(Sharp gradient estimate)} we have 
		\begin{equation}\label{sharp constant}
		|\nabla \bist_x|(y)\le c_N \nu_x^\frac{1}{N-2}
		\end{equation} for any $y \in X \setminus \{x\}$, where $c_N:=(N(N-2))^{\frac{1}{N-2}}$ and $|\nabla \bist_x|$ has the canonical pointwise representative:
		\item{(Rigidity)} $(X,\dist, \meas, x)$ is isomorphic to the $N$-metric measure cone with the pole over an $\RCD(N-2, N-1)$ space with $\bist_x=c_N  \nu_x^\frac{1}{N-2}\dist_x$ if the equality of (\ref{sharp constant}) holds for some $y \in X \setminus \{x\}$.
	\end{enumerate}
	\end{theorem}
	The proof follows from similar arguments as done in \cite{C12} via the Bochner inequalities (\ref{isnshshshs}) and (\ref{bochner}), however new technical difficulties appear because of the lack of the smoothness. It is worth mentioning that the non-linear potential theory on metric measure spaces helps us to overcome such difficulties. See \cite{BjornBjorn}.
	
	In connection with the rigidity of Theorem \ref{rcdsharp}, it is not hard to check that if $(X, \dist, \meas, x)$ is isomorphic to the $N$-metric measure cone with the pole over an $\RCD(N-2, N-1)$ space, then the heat kernel $p$ satisfies
\begin{equation}\label{asha89ryhabsjjs}
p(x, y, t)=\frac{2^{1-N}}{N \Gamma \left(\frac{N}{2}\right)  \meas(B_1(x))  t^{\frac{N}{2}}}\exp \left(-\frac{\dist(x, y)^2}{4t}\right),\quad \text{for all $y \in X$ and $t>0$,}
\end{equation}
thus whenever $(X, \dist, \meas)$ is non-parabolic, we have
\begin{equation}
G(x, y)=\frac{1}{N(N-2)\meas (B_1(x))}\dist(x, y)^{2-N}
\end{equation}
namely $\bist_x=c_N \nu_x^{\frac{1}{N-2}}\dist_x$.

\subsection{Proof of Theorem \ref{thm:smooth almost}}
	Under the preparation above, let us prove Theorem \ref{thm:smooth almost} via a contradiction. The proof is based on Cheeger-Colding theory \cite{CheegerColding, CheegerColding1, CheegerColding2, CheegerColding3} (see \cite{D}). 
	Thus assume that the assersion is not satisfied. Then there exist:
\begin{itemize}
\item a sequence of pointed non-parabolic Riemannian manifolds $(M_i^n, \dist^{g_i}, \Vol^{g_i}, x_i)$ of dimension $n$ with non-negative Ricci curvature $\mathrm{Ric}^{g_i}\ge 0$ pmGH converging to a pointed non-parabolic, non-collapsed $\RCD(0, n)$ space $(X, \dist, \haus^n, x)$:
\begin{equation}
(M_i^n, \dist^{g_i}, \mathrm{Vol}^{g_i}, x_i) \stackrel{\mathrm{pmGH}}{\to} (X, \dist, \haus^n, x),
\end{equation}
where $(M_i^n, \dist^{g_i}, \Vol^{g_i}, x_i)$ are away from $(\mathbb{R}^n, \dist_{\mathrm{Euc}}, \mathcal{L}^n, 0_n)$, namely $(X, \dist, \haus^n, x)$ is not isomorphic to $(\mathbb{R}^n, \dist_{\mathrm{Euc}}, \mathcal{L}^n, 0_n)$:
\item a sequence of points $y_i \in M^n_i$ converging to $y \in X$ with $x \neq y$ and $|\nabla \bist_{x_i}^{g_i}|(y_i)\to c_n  \omega_n^{\frac{1}{n-2}}$.
\end{itemize}
Recalling the subharmonicity of $G^{g_i}_{x_i}|\nabla \bist_{x_i}^{g_i}|^2$ as pointed out in subsection \ref{subColding}, the weak Harnack inequality (see for instance \cite{BjornBjorn}) allows us to conclude for a small $r>0$;
\begin{equation}
\int_{B_r(y_i)} \left||\nabla \bist_{x_i}^{g_i}|-c_n \omega_n^{\frac{1}{n-2}}\right|\di \mathrm{Vol}^{g_i} \to 0.
\end{equation}
This implies $|\nabla \bist_x|(z)=c_n \omega_n^{\frac{1}{n-2}}$ with $\nu_z =\omega_n$ for $\haus^n$-a.e. $z \in B_r(y)$ because of the Bishop inequality (\ref{asasj8a8s}). Therefore we can apply the rigidity result in Theorem \ref{rcdsharp} to get that  $(X, \dist, \haus^n, x)$ is isomorphic to $(\mathbb{R}^n, \dist_{\mathrm{Euc}}, \mathcal{L}^n, 0_n)$ because of the rigidity of the Bishop inequality (see \cite[Corollary 1.7]{DePhillippisGigli}). This is a contradiction. Thus we have Theorem \ref{thm:smooth almost}.

In order to justify the arguments above, combining with (\ref{s8syshsbb}), we immediately used the following compactness/continuity results for Green functions which may have independent interests for readers.
	
	\begin{proposition}[Compactness/continuity of Green functions]\label{contigreeeeen}
	Let
	$
	(X_i, \dist_i, \meas_i, x_i)
	$
	be a sequence of pointed non-parabolic $\RCD(0, N)$ spaces for some finite $N >2$. Assume that there exists $\phi \in L^1([1, \infty), \mathcal{L}^1)$ such that  for any $i$ we have 
	\begin{equation}\label{es8snsbsh}
	\frac{r}{\meas_i(B_r(x_i))}\le \phi(r),\quad \text{for $\mathcal{L}^1$-a.e. $r \in [1, \infty)$.}
	\end{equation}
	Then after passing to a subsequence there exists a pointed non-parabolic $\RCD(0,N)$ space $(X, \dist, \meas, x)$ such that $(X_i, \dist_i, \meas_i, x_i)$ pointed measured Gromov-Hausdorff converge to $(X, \dist, \meas, x)$ and that 
	\begin{equation}\label{cont green}
	G^{X_i}(x_i, y_i) \to G^X(x, y),\quad \text{for all $x_i, y_i \in X_i \to x, y \in X$, respectively with $x \neq y$.}
	\end{equation}
	\end{proposition}
	It should be emphasized that by the proof above, Theorem \ref{thm:smooth almost} is also justified even if we replace the manifold by a non-collapsed $\RCD(0, n)$ space. Compare with the proof of Theorem \ref{almost rigi}. Note that the assumption (\ref{es8snsbsh}) is essential, see \cite{HP} for the details.


\section{Related rigidity and almost rigidity results}\label{sec3}
In this section let us discuss related rigidity and almost rigidity results.
\subsection{Rigidity}\label{new rigid}
Let $(M^n, g)$ be a complete Riemannian manifold of dimension $n$ with non-negative Ricci curvature $\mathrm{Ric}^g\ge 0$. Then the \textit{Cheeger-Yau inequality} established in \cite{CheegerSTYau} implies
\begin{equation}\label{chya}
p(x,y,t)\ge p^{\mathbb{R}^n}(\dist^g(x, y), t)
\end{equation}
for all $x, y \in M^n$ and $t>0$, where $p^{\mathbb{R}^n}$ denotes the heat kernel of $\mathbb{R}^n$, namely
\begin{equation}\label{heateuc}
p^{\mathbb{R}^n}(\bar x, \bar y, t)=p^{\mathbb{R}^n}(|\bar x- \bar y|, t)=(4\pi t)^{-\frac{n}{2}}\exp \left( -\frac{|\bar x - \bar y|^2}{4t}\right).
\end{equation}
Let us introduce another rigidity result about $\bist_x^g$.
\begin{proposition}[Sharp upper bound and rigidity]\label{sharpbest}
Assume that $(M^n, g)$ is non-parabolic with $n \ge 3$ and $\mathrm{Ric}^g \ge 0$. Then for any fixed $x \in M^n$:
\begin{enumerate}
		\item{(Sharp upper bound)} we have 
		\begin{equation}\label{green comp}
G^g(x, y) \ge G^{\mathbb{R}^n}(\dist^g(x, y)).
\end{equation}
In other words,
\begin{equation}\label{bist comp}
\bist_x^g \le c_n \omega_n^{\frac{1}{n-2}}\dist_x^g:
\end{equation}
		\item{(Rigidity)} $(M^n, g)$ is isometric to $(\mathbb{R}^n, \dist_{\mathrm{Euc}})$ with $\bist_x^g=c_n  \omega_n^{\frac{1}{n-2}}\dist_x^g$  if the equality in (\ref{green comp}) (thus equivalently in (\ref{bist comp})) holds for some $y \in M^n\setminus \{x\}$.
	\end{enumerate}
\end{proposition}
\begin{proof}
Firstly let us provide two proofs of (1). 

The first one is simple, namely, integrating with respect to $t$ over $[0, \infty)$ in (\ref{chya}) implies (\ref{green comp}).

The second one is to apply (\ref{sharp constant2}), namely  letting $\bist^g_x(x):=0$, (\ref{sharp constant2}) shows that $\bist^g_x$ is $c_n \omega_n^{\frac{1}{n-2}}$-Lipschitz on $M^n$. Thus
\begin{equation}
\bist^g_x(y)=\bist^g_x(y)-\bist^g_x(x) \le c_n  \omega^{\frac{1}{n-2}}\dist^g(x,y),\quad \text{for any $y \in M^n$}
\end{equation}
which proves (\ref{bist comp}). Thus we have (1).

Next let us provide a proof of (2).

Fix a point $x \in M^n$ and assume that the equality in (\ref{green comp}) is attained at some $y \in M^n \setminus \{x\}$. Since the Laplacian comparison theorem states that $G^{\mathbb{R}^n}(\dist^g(x, \cdot))$ is subharmonic on $M^n \setminus \{x\}$, the function $G^g(x, \cdot)-G^{\mathbb{R}^n}(\dist^g(x, \cdot))$ is superharmonic on $M^n \setminus \{x\}$ because of the harmonicity of $G^g(x, \cdot)$ on $M^n \setminus \{x\}$. Thus applying the strong maximum principle shows $G^g(x, \cdot)=G^{\mathbb{R}^n}(\dist^g(x, \cdot))$ on $M^n \setminus \{x\}$. In particular since $|\nabla \bist_x^g|\equiv c_n \omega^{\frac{1}{n-2}}$ holds on $M^n \setminus \{x\}$, the rigidity result of Theorem \ref{smooth result} yields (2).\footnote{There exists another proof based on (\ref{chya}) as follows. The identity $G^g(x, \cdot)=G^{\mathbb{R}^n}(\dist^g(x, \cdot))$ on $M^n \setminus \{x\}$ with (\ref{chya}) shows that $p(x, y, t)=p^{\mathbb{R}^n}(\dist^g(x,y), t)$ for all $x, y \in X$ and $t>0$. Then we can apply a result in \cite{CT} to conclude.}
\end{proof}
By the proof above, if we wish to discuss a generalization of Proposition \ref{sharpbest} to $\RCD(0,N)$ spaces, it is natural to ask whether a Cheeger-Yau type inequality is valid or not for $\RCD(0,N)$ spaces, though it might be hard to adopt the original proof in \cite{CheegerSTYau} for the purpose. 
However we can apply a Harnack inequality established in \cite{GarofaloMondino, Jiang15} to realize this. Note that (\ref{9999}) below is sharp even for \textit{collapsed} $\RCD(0, N)$ spaces because the equality is attained for the $\RCD(0, N)$ space defined in (\ref{halfline}) with $x=0$ (see also (\ref{asha89ryhabsjjs})). Furthermore notice that roughly speaking (\ref{9999}) means that the heat kernel is bounded below by the one of the tangent cones at $x$ and $y$ because of (\ref{asha89ryhabsjjs}). 
\begin{proposition}[Cheeger-Yau inequality for $\RCD(0,N)$ spaces]\label{CheegerYau}
Let $(X, \dist, \meas)$ be an $\RCD(0, N)$ space for some finite $N \ge 1$. Then for all $x, y \in X$ and $t>0$
\begin{equation}\label{9999}
p(x, y, t) \ge \frac{2^{1-N}}{\min \{\nu_x, \nu_y\}N \Gamma \left( \frac{N}{2}\right) t^{\frac{N}{2}}} \exp \left(-\frac{\dist(x,y)^2}{4t}\right),
\end{equation}
where note that the right-hand-side of (\ref{9999}) coincides with
\begin{equation}
\frac{\omega_N}{\min \{\nu_x, \nu_y\}}p^{\mathbb{R}^N}(\dist(x, y), t)
\end{equation}
if $N$ is an integer because of (\ref{heateuc}).
In particular if $(X, \dist, \meas)$ is non-collapsed, then 
\begin{equation}\label{nsbsux7u}
p(x,y,t)\ge p^{\mathbb{R}^N}\left(\dist(x,y), t\right).
\end{equation}
\end{proposition} 
\begin{proof}
Note that with no loss of generality we can assume $\nu_x<\infty$.
Recalling \cite[Corollary 6.1]{Jiang15}, we know for all $x, y, z \in X$ and $0<s<t$,
\begin{equation}
p(x, z, s) \le p(y, z, t) \exp \left(\frac{\dist(x, y)^2}{4(t-s)}\right)  \left(\frac{t}{s}\right)^{\frac{N}{2}}.
\end{equation}
Thus letting $x=z$, we have
\begin{equation}\label{s6sgg}
s^{\frac{N}{2}}p(x, x, s)  \exp \left(-\frac{\dist(x, y)^2}{4(t-s)}\right) t^{-\frac{N}{2}} \le p(y, x, t).
\end{equation}
Recalling that any tangent cone at $x$ is the $N$-metric measure cone over an $\RCD(N-2, N-1)$ space,
thanks to the continuity of the heat kernels with respect to the pmGH convergence established in \cite{AmbrosioHondaTewodrose, ZhangZhu} after \cite{GigliMondinoSavare13} with (\ref{asha89ryhabsjjs}), we know 
\begin{equation}\label{s6sgg2}
s^{\frac{N}{2}}p(x, x, s) =\frac{s^{\frac{N}{2}}}{\meas (B_{\sqrt{s}}(x))} \cdot \meas (B_{\sqrt{s}}(x)) p(x,x,s) \to \frac{1}{\nu_x} \cdot \frac{2^{1-N}}{N \Gamma \left( \frac{N}{2}\right)},\quad \text{as $s \to 0^+$,}
\end{equation}
where we used a fact that considering a rescaled pointed $\RCD(0, N)$ space:
\begin{equation}
\left(X, \frac{1}{\sqrt{s}}\dist, \frac{1}{\meas (B_{\sqrt{s}}(x))}\meas, x \right),
\end{equation}
after passing to a subsequence as $s \to 0^+$, $\meas (B_{\sqrt{s}}(x)) p(x,x,s)$ converges to the one of the tangent cone, thus the limit is equal to $\frac{2^{1-N}}{N \Gamma \left( \frac{N}{2}\right)}$ because of (\ref{asha89ryhabsjjs}).
Thus letting $s \to 0^+$ in (\ref{s6sgg}) with (\ref{s6sgg2}) completes the proof of (\ref{9999}) because of the symmetry of the heat kernel, $p(x, y, t)=p(y,x, t)$.

The remaining statement, (\ref{nsbsux7u}), is a direct consequence of (\ref{9999}) with the Bishop inequality (\ref{asasj8a8s}).
\end{proof}

Finally as a corollary of Proposition \ref{CheegerYau}, we obtain a non-smooth analogue of Proposition \ref{sharpbest}. Since the proof is similar to that of Proposition \ref{sharpbest} with the Laplacian comparison theorem established in \cite{Gigli1}, we omit it.
\begin{corollary}[Sharp upper bound and rigidity]\label{sharpbest2}
Let $(X, \dist, \meas)$ be a non-parabolic $\RCD(0, N)$ space for some finite $N >2$. Then for any $x \in X$ with $\nu_x<\infty$:
\begin{enumerate}
\item{(Sharp upper bound)} we have 
\begin{equation}\label{green comp rcd}
G(x,y) \ge \frac{1}{N(N-2)\nu_x} \dist(x, y)^{2-N},\quad \text{for any $y \in X$.}
\end{equation}
In other words,
\begin{equation}\label{bist comp44}
\bist_x \le c_N\nu_x^{\frac{1}{N-2}}\dist_x:
\end{equation}
		\item{(Rigidity)} $(X, \dist, \meas, x)$ is isomorphic to the $N$-metric measure cone with the pole over an $\RCD(N-2, N-1)$ space with $\bist_x=c_N  \nu_x^\frac{1}{N-2}\dist_x$  if  the equality in (\ref{green comp rcd}) holds for some $y \in X \setminus \{x\}$.
	
\end{enumerate}
\end{corollary}

\subsection{Almost rigidity for $\RCD(0, N)$ spaces}\label{subsec2}

We can prove an analogous almost rigidity result of Theorem \ref{thm:smooth almost}, for $\bist_x$ instead of $|\nabla \bist_x|$. Since this can be stated for $\RCD(0, N)$ spaces (recall that Theorem \ref{thm:smooth almost} can be also stated for $\RCD(0, N)$ spaces, see \cite{HP}), let us introduce it as follows.
\begin{theorem}[Almost rigidity, Part II]\label{almost rigi}
For any integer $n \ge 3$, all $0<\eps<1, 0<r<R, 1 \le p <\infty$ and $\phi \in L^1([1, \infty), \mathcal{L}^1)$ there exists $\delta:=\delta(n, \eps, r, R, p, \phi)>0$ such that 
if 
a non-collapsed $\RCD(0, n)$ space $(X, \dist, \haus^n)$ satisfies
\begin{equation}
\frac{s}{\haus^n (B_s(x))} \le \phi(s),\quad \text{for $\mathcal{L}^1$-a.e. $s \in [1, \infty)$}
\end{equation}
for some $x \in X$ and that
\begin{equation}\label{almost max point2}
c_n\omega_n^{\frac{1}{n-2}}\dist_x(y)-\bist_x(y)\le \delta
\end{equation}
holds for some $y \in B_R(x) \setminus B_r(x)$, then we have 
\begin{align}\label{r lower2}
\dist_{\mathrm{pmGH}}\left( (X, \dist, \haus^n, x), (\mathbb{R}^n, \dist_{\mathbb{R}^n}, \mathcal{L}^n, 0_n) \right)\le \eps
\end{align}
and 
\begin{equation}\label{impsosirbb2}
\left\| \bist_x-c_n\omega_n^{\frac{1}{n-2}}\dist_x\right\|_{L^{\infty}(B_R(x))}+ \left\| \bist_x-c_n\omega_n^{\frac{1}{n-2}}\dist_x\right\|_{H^{1, p}(B_R(x))}\le \eps.
\end{equation}
\end{theorem}
\begin{proof}
The proof is very similar to that of Theorem \ref{thm:smooth almost}.
Namely assume that the assertion is not satisfied. Then there exist:
\begin{itemize}
\item a sequence of pointed non-parabolic, non-collapsed $\RCD(0,  n)$ spaces $(X_i, \dist_i, \haus^n, x_i)$ pmGH converging to a pointed non-parabolic, non-collapsed $\RCD(0, n)$ space $(X, \dist, \haus^n, x)$:
\begin{equation}
(X_i, \dist_i, \haus^n, x_i) \stackrel{\mathrm{pmGH}}{\to} (X, \dist, \haus^n, x)
\end{equation}
with the locally uniform and $H^{1,2}_{\mathrm{loc}}$-strong convergence of $\bist_{x_i}$ to $\bist_x$, 
where $(X_i, \dist_i, \haus^n_i, x_i)$ are away from $(\mathbb{R}^n, \dist_{\mathrm{Euc}}, \mathcal{L}^n, 0_n)$, namely $(X, \dist, \haus^n, x)$ is not isometric to $(\mathbb{R}^n, \dist_{\mathrm{Euc}}, \mathcal{L}^n, 0_n)$:
\item a sequence of points $y_i \in X_i$ converging to $y \in X$ with $x \neq y$ and $c_n\omega_n^{\frac{1}{n-2}}\dist_{x_i}(y_i) - \bist_{x_i}(y_i) \le \delta_i$ for some $\delta_i \to 0^+$.
\end{itemize}
It follows from (\ref{asasj8a8s}) and (\ref{bist comp44}) that
\begin{equation}\label{ee88jj}
0\le  c_n \nu_{x_i}^{\frac{1}{n-2}} \dist_{x_i}(y_i)-\bist_{x_i}(y_i) \le c_n \omega_n^{\frac{1}{n-2}} \dist_{x_i}(y_i)-\bist_{x_i}(y_i) \le \delta_i \to 0,
\end{equation}
thus Proposition \ref{contigreeeeen} yields $\bist_x(y)=c_n\nu_x^{\frac{1}{n-2}}\dist_x(y)$ and $\nu_x=\omega_n$. 
In particular $x$ is an $n$-regular point because of the rigidity of the Bishop inequality (\ref{asasj8a8s}). Thus applying the rigidity result of Corollary \ref{sharpbest2} proves that  $(X, \dist, \haus^n, x)$ is isometric to $(\mathbb{R}^n, \dist_{\mathrm{Euc}}, \mathcal{L}^n, 0_n)$. This is a contradiction. Thus we have Theorem \ref{almost rigi}.
\end{proof}



\begin{acknowledgement}
The author would like to thank the referee for his/her careful reading with valuable suggestions for a revision.
He acknowledges supports of the Grant-in-Aid for Scientific Research (B) of 20H01799, the
Grant-in-Aid for Scientific Research (B) of 21H00977 and Grant-in-Aid for Transformative
Research Areas (A) of 22H05105.
\end{acknowledgement}

\end{document}